\documentclass[11pt,twoside, reqno]{amsart}

\usepackage{amsmath}
\usepackage{amsthm}
\usepackage{amsfonts, amssymb}
\usepackage{mathrsfs}
\usepackage[all]{xy}
\usepackage{url}

\usepackage{latexsym}
\usepackage[dvips]{graphicx}

\theoremstyle{plain}
\newtheorem{lema}{Lemma}
\newtheorem{prop}[lema]{Proposition}
\newtheorem{teo}[lema]{Theorem}

\theoremstyle{remark}

\theoremstyle{definition}

\begin{document}

\title[Spaces which invert weak homotopy equivalences]{Spaces which invert weak homotopy equivalences}

\author[J.A. Barmak]{Jonathan Ariel Barmak $^{\dagger}$}

\thanks{$^{\dagger}$ Researcher of CONICET. Partially supported by grant UBACyT 20020160100081BA}

\address{Universidad de Buenos Aires. Facultad de Ciencias Exactas y Naturales. Departamento de Matem\'atica. Buenos Aires, Argentina.}

\address{CONICET-Universidad de Buenos Aires. Instituto de Investigaciones Matem\'aticas Luis A. Santal\'o (IMAS). Buenos Aires, Argentina. }

\email{jbarmak@dm.uba.ar}

\begin{abstract}
It is well known that if $X$ is a CW-complex, then for every weak homotopy equivalence $f:A\to B$, the map $f_*:[X,A]\to [X,B]$ induced in homotopy classes is a bijection. For which spaces $X$ is $f^*:[B,X]\to [A,X]$ a bijection for every weak equivalence $f$? This question was considered by J. Strom and T. Goodwillie. In this note we prove that a non-empty space inverts weak equivalences if and only if it is contractible.  \end{abstract}



\maketitle

We say that a space $X$ \textit{inverts weak homotopy equivalences} if the functor $[-,X]$ inverts weak equivalences, that is, for every weak homotopy equivalence $f:A\to B$, the induced map $f^*:[B,X]\to [A,X]$ is a bijection. As usual $[A,X]$ stands for the set of homotopy classes of maps from $A$ to $X$. This property is clearly a homotopy invariant. In \cite{S} Jeff Strom asked for the characterization of such spaces. Tom Goodwillie observed that if $X$ inverts weak equivalences and is $T_1$ (i.e. its points are closed), then each path-component is weakly contractible (has trivial homotopy groups) and then contractible. His idea was to use finite spaces weak homotopy equivalent to spheres. A map from a connected finite space to a $T_1$-space has a connected and discrete image and is therefore constant. This is one of the many interesting applications of non-Hausdorff spaces to homotopy theory. Goodwillie also proved that if a space inverts weak equivalences, then it must be connected. In this note we follow his ideas and give a further application of non-Hausdorff spaces to obtain the expected characterization:

\begin{teo} \label{main}
A non-empty space $X$ inverts weak homotopy equivalences if and only if it is contractible.
\end{teo}

\begin{lema} [Goodwillie] \label{wc}
Suppose that $X$ inverts weak homotopy equivalences and is weakly contractible. Then it is contractible.
\end{lema}
\begin{proof}
Just take the weak homotopy equivalence $X\to *$.
\end{proof}

\begin{prop} [Goodwillie] \label{tom}
Let $X$ be a space which inverts weak homotopy equivalences. Then it is connected.
\end{prop}
\begin{proof}
We can assume $X$ is non-empty. Suppose that $X_0$ and $X_1$ are two path-components of $X$. Let $x_0\in X_0$ and $x_1\in X_1$. Let $A=\mathbb{N}_{0}$ be the set of nonnegative integers with the discrete topology and $B=\{0\}\cup \{\frac{1}{n}\}_{n\in \mathbb{N}} \subseteq \mathbb{R}$ with the usual topology. The map $f:A\to B$ which maps $0$ to $0$ and $n$ to $\frac{1}{n}$ for every $n$, is a weak homotopy equivalence. Take $g:A\to X$ defined by $g(0)=x_0$ and $g(n)=x_1$ for every $n\ge 1$. By hypothesis there exists a map $h:B\to X$ such that $h(0)\in X_0$ and $h(\frac{1}{n})\in X_1$ for every $n\ge 1$. Since $\frac{1}{n} \rightarrow 0$, $X_0$ intersects the closure of $X_1$. Thus $X_0$ and $X_1$ are contained in the same component of $X$.  
\end{proof}

\begin{lema} \label{exp}
Let $X$ be a space which inverts weak equivalences and let $Y$ be a locally compact Hausdorff space. Then the mapping space $X^Y$, considered with the compact-open topology, also inverts weak equivalences.
\end{lema}
\begin{proof}
This follows from a direct application of the exponential law and the fact that a weak equivalence $f:A\to B$ induces a weak equivalence $f\times 1_Y:A\times Y \to B\times Y$.
\end{proof}

By Lemmas \ref{wc} and \ref{exp} it only remains to show that a map that inverts weak equivalences is path-connected. If we require a slightly different property, this is easy to prove using only Hausdorff spaces. The following result is not needed for the proof of Theorem \ref{main}.

\begin{prop} \label{path}
Let $(X,x_0)$ be a pointed space such that for every weak homotopy equivalence $f:A\to B$ between Hausdorff spaces and every $a_0\in A$, the induced map $$f^*:[(B,f(a_0)),(X,x_0)]\to [(A,a_0),(X,x_0)]$$ is a bijection. Then $X$ is path-connected.
\end{prop}
\begin{proof}
Let $X_0$ be the path-component of $x_0$. Let $X_1$ be any path-component of $X$ and let $x_1\in X_1$. Let $A, B, f, g$ be as in Proposition \ref{tom}. Let $a_0=0\in A$. By hypothesis there exists $h:(B,0)\to (X,x_0)$ such that $hf\simeq g \textrm{ rel } \{0\}$. In particular $h(1)\in X_1$. Define $h':B\to X$ by $h'(0)=x_0$ and $h'(\frac{1}{n})=h(\frac{1}{n+1})$ for $n\ge 1$. The continuity of $h'$ follows from that of $h$. Since $h'(\frac{1}{n})=h(\frac{1}{n+1}) \in X_1$ and $h(\frac{1}{n})\in X_1$ for every $n\ge 1$, there exists a homotopy $H:A\times I\to Y$ from $f^*(h')$ to $f^*(h)$. Moreover we can take $H$ to be stationary on $0\in A$. Since $f^*:[(B,0),(X,x_0)]\to [(A,0),(X,x_0)]$ is injective, there exists a homotopy $F:B\times I \to X$, $F:h'\simeq h \textrm{ rel } \{0\}$. The map $F$ gives a collection of paths from $h(\frac{1}{n+1})$ to $h(\frac{1}{n})$. We glue all these paths to form a path from $x_0$ to $h(1)$. That is, define $\gamma :I\to X$ by $\gamma(0)=x_0$ and $\gamma (t)=F(\frac{1}{n},(\frac{1}{n}-\frac{1}{n+1})^{-1}(t-\frac{1}{n+1}))$ if $t\in [\frac{1}{n+1}, \frac{1}{n}]$. Note that $\gamma$ is continuous in $t=0$ for if $U\subseteq X$ is a neighborhood of $x_0$, then $\{0\}\times I \subseteq F^{-1}(U)$, and by the tube lemma there exists $n_0\ge 1$ such that $\{\frac{1}{n}\}\times I\subseteq F^{-1}(U)$ for every $n\ge n_0$. Then $[0,\frac{1}{n_0}]\subseteq \gamma ^{-1}(U)$. Hence, $x_0$ and $h(1)$ lie in the same path-component, so $X_0=X_1$.  
\end{proof}

Note that if a contractible space $X$ satisfies the hypothesis of Proposition \ref{path} for some point $x_0$, then by taking $A=B=X$, $a_0=x_0$ and $f$ the constant map $x_0$, one obtains that $\{x_0\}$ is a strong deformation retract of $X$. Conversely, a based space $(X,x_0)$ such that $\{x_0\}$ is a strong deformation retract of $X$, clearly satisfies the hypothesis of the proposition.

\bigskip

The following result is the key lemma for proving Theorem \ref{main} and in contrast to the previous result, the proof provided uses non-Hausdorff spaces.

\begin{lema} \label{path2}
Let $X$ be a space which inverts weak homotopy equivalences. Then $X$ is path-connected.
\end{lema}
\begin{proof}
We can assume $X$ is non-empty. Let $X_0$ and $X_1$ be path-components of $X$. Let $B$ be a set with cardinality $\# B > \alpha= \max \{\#X,c\}$. Here $c$ denotes the cardinality $\# \mathbb{R}$ of the continuum. Consider the following topology in $B$: a proper subset $F\subseteq B$ is closed if and only if $\# F\le \alpha$. Note that the path-components of $B$ are the singletons, for if $\gamma :I\to B$ is a path, then its image has cardinality at most $\alpha$, so it is connected and discrete and then constant. Let $A$ be the discretization of $B$, i.e. the same set with the discrete topology. Then the identity $id:A\to B$ is a weak homotopy equivalence. Let $b_0$ and $b_1$ be two different points of $B$. Define $g:A\to X$ in such a way that $g(b_0)\in X_0$ and $g(b_1)\in X_1$ (define $g$ arbitrarily in the remaining points of $A$). Then $g$ is continuous. Since the identity $id^*:[B,X]\to [A,X]$ is surjective, there exists a map $h:B\to X$ such that $h\circ id\simeq g$. In particular $h(b_0)\in X_0$ and $h(b_1)\in X_1$. Since $\#B> \alpha \ge \# X$ and $B =\bigcup\limits_{x\in X} h^{-1}(x)$, there exists $x\in X$ such that $\#h^{-1}(x)>\alpha$. Let $U\subseteq X$ be an open neighborhood of $h(b_0)$. Then $h^{-1}(U^c)\subseteq B$ is a proper closed subset, so $\# h^{-1}(U^c) \le \alpha$. Thus, $h^{-1}(x)$ is not contained in $h^{-1}(U^c)$ and then $x\in U$. Since every open neighborhood of $h(b_0)$ contains $x$, there is a continuous path from $x$ to $h(b_0)$, namely $t\mapsto x$ for $t<1$ and $1\mapsto h(b_0)$. In particular $x\in X_0$. Symmetrically, $x\in X_1$. Therefore $X_0=X_1$.     
\end{proof}

\textit{Proof of Theorem \ref{main}}. It is clear that a contractible space inverts weak equivalences. Suppose $X\neq \emptyset$ is a space which inverts weak equivalences. By Lemma \ref{exp}, $X^{S^n}$ inverts weak equivalences for every $n\ge 0$ and then it is path-connected. Therefore $\pi_n(X)$ is trivial for every $n\ge 0$ and by Lemma \ref{wc}, $X$ is contractible.  


%
%

\end{document}